\documentclass[leqno]{amsart}
\usepackage{caption}
\usepackage{amsmath,amssymb,amsthm,mathtools,mathrsfs,tikz}
\usepackage{environ}
\usepackage{stmaryrd}
\usetikzlibrary{arrows}
\usetikzlibrary{positioning}
\usetikzlibrary{decorations.text}
\usetikzlibrary{decorations.pathmorphing}
\usetikzlibrary{decorations.pathreplacing}
\makeatletter
\newsavebox{\@brx}
\newcommand{\llangle}[1][]{\savebox{\@brx}{\(\m@th{#1\langle}\)}%
	\mathopen{\copy\@brx\kern-0.5\wd\@brx\usebox{\@brx}}}
\newcommand{\rrangle}[1][]{\savebox{\@brx}{\(\m@th{#1\rangle}\)}%
	\mathclose{\copy\@brx\kern-0.5\wd\@brx\usebox{\@brx}}}
\makeatother
\makeatletter
\newsavebox{\measure@tikzpicture}
\NewEnviron{scaletikzpicturetowidth}[1]{%
	\mathrm{Def}\tikz@width{#1}%
	\mathrm{Def}\tikzscale{1}\begin{lrbox}{\measure@tikzpicture}%
		\BODY
	\end{lrbox}%
	\pgfmathparse{#1/\wd\measure@tikzpicture}%
	\edef\tikzscale{\pgfmathresult}%
	\BODY
}
\makeatother
\usepackage{bbm}
\DeclarePairedDelimiter\norm{\lvert}{\rvert}

\usepackage{bm}
\usepackage[colorlinks=true,linkcolor=blue,citecolor=blue]{hyperref}
\usepackage{enumitem}
\usepackage{csquotes}

\def\irr#1{{\rm  Irr}(#1)}

\newcounter{intro}

\newtheorem{introthm}[intro]{Theorem}

\newtheorem{thm}{Theorem}
\newtheorem{lem}[thm]{Lemma}

\theoremstyle{remark}
\theoremstyle{definition}

\title{A characterization of Nested Groups in terms of conjugacy classes}

\author{Shawn T. Burkett}
\address{Department of Mathematical Sciences, Kent State University, Kent,
	Ohio 44240, U.S.A.} \email{sburket1@kent.edu}

\author{Mark L. Lewis}
\address{Department of Mathematical Sciences, Kent State University, Kent,
	Ohio 44240, U.S.A.} \email{lewis@math.kent.edu}

\date{\today}

\begin{document}

\begin{abstract}
A group is nested if the centers of the irreducible characters form a chain.  In this paper, we will show that there is a set of subgroups associated with the conjugacy classes of group so that a group is nested if and only if these subgroups form a chain.
\end{abstract}

\maketitle

Throughout this paper all groups will be finite.  In Research problems and themes I of \cite{YBGPPOV1}, Berkovich suggests the following two problems: Problem \#24 asks for a description of the $p$-groups $G$ for which the centers (quasi-kernels) of the irreducible characters form a chain with respect to inclusion and Problem \#30 suggests the study of the $p$-groups $G$ such that $\chi(1)^2 = \norm{G:Z(\chi)}$ for every character $\chi \in \irr G$.  In \cite{AN12gvz} and \cite{AN16gvz}, Nenciu initiates the study of groups that satisfy both of these properties and obtains several results.  Nenciu uses the term nested for groups satisfying the following condition on the centers of its irreducible characters: For all $\psi,\chi\in\irr{G}$ with $\psi(1)\le\chi(1)$, we have $Z(\chi)\le Z(\psi)$ whenever $\psi(1)\le\chi(1)$. It is easy to see that in this case the centers of such groups form a chain with respect to inclusion.

In this paper, we follow \cite{ML19gvz} and say that a group $G$ is {\it nested } if for all characters $\chi, \psi \in \irr G$ either $Z(\chi) \le Z(\psi)$ or $Z (\psi) \le Z(\chi)$. We note that this definition of nested is somewhat weaker than the definition given by Nenciu, but in Lemma 7.1 of \cite{ML19gvz}, it is shown that this definition is equivalent to the definition of Nenciu in the situations studied in \cite{AN12gvz} and \cite{AN16gvz}.  The second author proved a number of results regarding nested groups in \cite{ML19gvz}. Among these results, Theorem 1.3 of \cite{ML19gvz} gives a characterization of nested groups that is character-free. 

Many people have observed that there is a strong connection between the irreducible characters of a group and the conjugacy classes of group.  David Chillag has outlined many of these connections in \cite{chillag}.  We have defined nested groups in terms the irreducible characters of the group.  It makes sense to ask if there is an equivalent definition in terms of the conjugacy classes of $G$.  In this paper, we present such a characterization.

We now introduce some notation.  Let $G$ be a group.  Define $\gamma_G(g) = \{ [g,x] \mid x \in G\}$ for every element $g\in G$.  We let $[g,G]$ denote the subgroup generated by $\gamma_G (g)$.   We believe that it is well-known that $[g,G]$ is normal in $G$, but we do include a proof of this as Lemma 2.2 of \cite{SBML}, and so, $[g,G]$ is determined by the conjugacy class of $g$.  With these definitions, we prove the following: 

\begin{introthm} \label{nested cl}
Let $G$ be a nonabelian group. Then $G$ is nested if and only if for all elements $g,h \in G \setminus Z(G)$, either $[g,G] \le [h,G]$ or $[h,G] \le [g,G]$.
\end{introthm}

Nenciu defines a group $G$ to be a {\it generalized VZ}-group (or a GVZ-group) if it satisfies $\chi(1)^2 = \norm {G:Z(\chi)}$ for every character $\chi \in \irr G$.  We obtain the following characterization of nested GVZ-groups.  As we stated above, Nenciu studied nested GVZ-groups and it makes sense to apply Theorem \ref{nested cl} in this case, and we obtain the following theorem. 

\begin{introthm} \label{nested gvz class}
Let $G$ be a nonabelian group. Then $G$ is a nested GVZ-group if and only if for all elements $g, h \in G$, either $\gamma_G (g) \subseteq \gamma_G (h)$ or $\gamma_G (h) \subseteq \gamma_G (g)$.
\end{introthm}

As we stated above, Nenciu introduced the study of nested GVZ-groups in her papers \cite{AN12gvz} and \cite{AN16gvz}.  Thus, from the beginning the study of nested groups and GVZ-groups has been linked.  However, we believe that the two concepts are really independent of each other.  In \cite{SBMLII} and \cite{ML19gvz}, we have focused on studying nested groups, and the examples presented in \cite{ML19gvz} show that there are nested groups that are not GVZ-groups.  That being said, both of those papers do include a significant number of results regarding nested GVZ-groups.  On the other hand, in the paper \cite{SBML} we study GVZ-groups.  We note that the main focus of this paper is on nested groups, and that other than one result, which is well-known, the results used to prove Theorem \ref{nested cl} are independent of the work in \cite{SBML}.  Not surprisingly, we use a characterization of GVZ-groups proved in \cite{SBML} in the proof of Theorem \ref{nested gvz class}, which is about nested GVZ-groups.  We also mention that the techniques in this paper are completely independent of the techniques used in \cite{SBMLII} and \cite{ML19gvz}.
  
In \cite{SBML}, we prove that $G$ is a GVZ-group if and only if  $\gamma_G (g)$ is a group for every element $g \in G$.  Note that we are not assuming this in Theorem \ref{nested gvz class}, instead this ends up being a consequence of the containment of the sets $\gamma_G (g)$ for $g\in G$. 

Recall from above that a group $G$ is nested if $\{ Z (\chi) \mid \chi \in \irr G\}$ is a chain with respect to inclusion.  In this case, we write $\{ Z (\chi) \mid \chi \in\irr G \} = \{ X_0, X_1,\dotsb, X_n \}$, where $G = X_0 > X_1 > \dotsb X_n \ge 1$. Following \cite {ML19gvz}, we call this the {\it chain of centers} for $G$. 


When $N$ is a normal subgroup of $G$, we write $\irr {G \mid N}$ for the set of characters in $\irr G$ that do not contain $N$ in their kernels.  The following lemma appears in \cite{SBML}.  Since the proof is only one line, we include it here also.

\begin{lem} \label{cen cond}
Let $G$ be a group.  Fix an element $g \in G$ and a character $\chi \in \irr G$.  Then $g \in Z(\chi)$ if and only if $[g,G] \le \ker (\chi)$.
\end{lem}

\begin{proof}
This follows immediately from the definition that states: $Z(\chi)/\ker (\chi) = Z (G/\ker (\chi))$.
\end{proof}

This next theorem includes Theorem \ref{nested cl}.  We also obtain a partial generalization of Nenciu's result \cite [Theorem 3.3]{AN16gvz}, which states that when $G$ is a nested group with chain of centers $G = X_0 > X_1 > \dotsb > X_n > 1$, then $G$ is a GVZ-group if and only if $\gamma_G (g) = [X_i,G]$ for every element $g \in X_i \setminus X_{i+1}$ and every integer $0 \le i \le n-1$.  In particular, we show when $G$ is a nested group, then $[g,G] = [X_i,G]$ for all $g \in X_i \setminus X_{i+1}$.  Hence, removing the GVZ hypothesis, we need to replace $\gamma_G (g)$ by $[g,G]$ to obtain equality with $[X_i,G]$.

\begin{thm}\label{nested chain}
Let $G$ be a nonabelian group. Then $G$ is nested if and only if the set $\{[g,G] : g \in G\}$ is a chain. Moreover, in the event that $G$ is nested with chain of centers $G = X_0 > X_1 > \dotsb > X_n \ge 1$, we have $[g,G] = [X_i,G]$ for every element $g \in X_i \setminus X_{i+1}$ and every integer $0 \le i \le n-1$.
\end{thm}

\begin{proof}
First, assume that $G$ is nested with chain of centers $G = X_0 > X_1 > \dotsb > X_n \ge 1$. We will show that $[g,G] = [X_i,G]$ for every element $g \in X_i \setminus X_{i+1}$ and every integer $i \ge 0$.  Thus, we fix an integer $i \ge 0$ and an element $g \in X_i \setminus X_{i+1}$.  If $\psi \in \irr {G \mid [X_i,G]}$, then $X_i$ is not contained in $Z(\psi)$ by Lemma \ref{cen cond}. Thus, we have $Z(\psi) \le X_{i+1}$. This implies that $g \notin Z(\psi)$ and hence also that $[g,G] \nleq \ker(\psi)$. We deduce that $[X_i,G] \le [g,G]$. But $g \in X_i$, so $[g,G] \le [X_i,G]$ as well; this gives $[g,G] = [X_i,G]$. It follows that the set $\{ [g,G] \mid g \in G\}$ is a chain and that $[g,G] = [X_i,G]$ for every element $g \in X_i \setminus X_{i+1}$ and every integer $0 \le i\le n-1$.
	
Now assume that the set $\{[g,G] \mid g \in G\}$ is a chain.  For each element $g \in G$, define the subgroup $Z_g \le G$ by $Z_g/[g,G] = Z (G/[g,G])$.  Fix an element $g \in G$.  Then we have $[Z_g,G] \le [g,G]$. On the other hand, since $[g,G] \le [g,G]$, we have $g \in Z_g$ and so, $[g,G] \le [Z_g,G]$.  This proves that $[Z_g,G] = [g,G]$ for every element $g \in G$.  Also note that the set $\{Z_g \mid g \in G\}$ must also be a chain.  By Lemma 3.2 of \cite{ML19gvz}, it suffices to show that there exists an element $h \in G$ so that $G = Z_h$ and for every normal subgroup $N \le G$, there exists an element $g \in G$ so that $Z_N = Z_g$.  First, there exists an element $h \in G$ so that $[h,G]$ is maximal among the sets $[g,G]$.  Since these subgroups form a chain, we have $[g,G] \le [h,G] \le G'$ for every element $g \in G$.  Observe that $G'$ is generated by the various subgroups $[g,G]$ as $g$ runs over $G$, and so, we conclude that $G' = [h,G]$.  Notice that $[h,G] = G'$ implies that $Z_h = G$.
	
Consider a normal subgroup $N < G$.  Set $Z_N/N = Z(G/N)$.  We can find an element $g \in G$ so that $Z_N \le Z_g$ and $Z_N$ is not contained in $Z_k$ where $k$ is the element of $G$ so that $Z_k < Z_g$ are consecutive terms in the chain of $Z_g$'s.  Hence, there is an element $x \in Z_N \setminus Z_k$.  We have that $[x,G] \not\le [k,G]$ and $[x,G] \le [g,G]$.  It follows that $[x,G] = [g,G]$.  Since $x \in Z_N$, we have $[x,G] \le N$.  Thus, if $y \in Z_g$, then $[y,G] \le [g,G] = [x,G] \le N$ and thus, $y \in Z_N$.  We conclude that $Z_g \le Z_N$, and thus, $Z_N = Z_g$ as desired.  
\end{proof}

We obtain Theorem \ref{nested gvz class}.  Notice that we recover one implication of Nenciu's result \cite [Theorem 3.3]{AN16gvz}.

\begin{thm}\label{nested gvz thm}
Let $G$ be a nonabelian group. The following are equivalent. 
\begin{enumerate}[label={\bf(\arabic*)}]
\item $G$ is a nested GVZ-group. 
\item For every $g, h \in G$, either $\gamma_G (g) \subseteq \gamma_G (h)$ or $\gamma_G (h) \subseteq \gamma_G (g)$.
\end{enumerate}
Moreover, in the event that $G$ is a nested GVZ-group with chain of centers $G = X_0 > X_1 > \dotsb > X_n > 1$, then we have $\mathrm{cl}_G (x) = x[X_i,G]$ and $[x,G] = [X_i,G]$ for every integer $0 \le i\le n$ and every element $x \in X_i \setminus X_{i+1}$.
\end{thm}

\begin{proof}
We first assume (1).  Since $G$ is a GVZ-group, we know that $\gamma_G (g)$ is a subgroup of $G$ for all elements $g \in G$ by Theorem B of \cite{SBML} and thus, $\gamma_G (g) = [g,G]$.  Now conclusion (2) follows immediately from Theorem \ref{nested cl}.  Now assume (2). Let $g, y \in G$. Then $\gamma_G (g) = \gamma_G (g^y)$, since these sets have the same size. Then means that for any element $x \in G$, we have
\[ [g,y]^{-1} [g,x] = [y,g] [g,y(y^{-1}x)] = [g^y,y^{-1}x] \in \gamma_G(g). \]
It follows that $\gamma_G (g)$ is a subgroup of $G$.  Hence, $G$ is a GVZ-group by Theorem B of \cite{SBML}.  It now follows that $\{[g,G] \mid g \in G\}$ is a chain, and so, $G$ is nested by Theorem~\ref{nested cl} (2).	
\end{proof}

\end{document}